\documentclass[11pt]{article}
\usepackage[a4paper]{anysize}\marginsize{1.5cm}{1.5cm}{1.5cm}{1cm}
\pdfpagewidth=\paperwidth \pdfpageheight=\paperheight
\usepackage{amsfonts,amssymb,amsthm,amsmath,eucal,tabu,url}
\usepackage{pgf}
 \usepackage{array}
 \usepackage{pstricks}
 \usepackage{pstricks-add}
 \usepackage{pgf,tikz}
 \usetikzlibrary{automata}
 \usetikzlibrary{arrows}
 \usepackage{indentfirst}
\theoremstyle{plain}
\newtheorem{thm}{Theorem}[section]
\newtheorem{theorem}[thm]{Theorem}

\newtheorem{corollary}[thm]{Corollary}
\newtheorem{conjecture}[thm]{Conjecture}

\theoremstyle{definition}
\newtheorem{definition}[thm]{Definition}
\newtheorem{remark}[thm]{Remark}
\newtheorem{example}[thm]{Example}

\newtheorem{question}[thm]{Question}

\newtheorem{thevarthm}[thm]{\varthmname}

\newenvironment{varthm*}[1]{\trivlist\item[]{\bf #1.}\it}{\endtrivlist}

\renewcommand\geq{\geqslant}

\renewcommand\leq{\leqslant}

\newcommand\be{\begin{eqnarray*}}
\newcommand\ee{\end{eqnarray*}}

\newcommand\newop[2]{\def#1{\mathop{\rm #2}\nolimits}}
\newop\log{log}
\newop\ord{ord}
\newop\Gal{Gal}
\newop\SL{SL}
\newop\Bl{Bl}
\newop\mult{mult}
\newop\mass{mass}
\newop\div{div}
\newop\codim{codim}
\newop\sing{sing}
\newop\vdim{vdim}
\newop\edim{edim}
\newop\Ass{Ass}
\newop\size{size}
\newop\reg{reg}
\newop\satdeg{satdeg}
\newop\supp{supp}
\newop\Neg{Neg}
\newop\Nef{Nef}
\newop\Nefh{Nef_H}
\newop\Eff{Eff}
\newop\Zar{Zar}
\newop\MB{MB}
\newop\MBxC{MB\mathit{(x,C)}}
\newop\NnB{NnB}
\newop\Bigg{Big}
\newop\Effbar{\overline{\Eff}}

\def\keywordname{{\bfseries Keywords}}%
\def\keywords#1{\par\addvspace\medskipamount{\rightskip=0pt plus1cm
\def\and{\ifhmode\unskip\nobreak\fi\ $\cdot$
}\noindent\keywordname\enspace\ignorespaces#1\par}}
\def\subclassname{{\bfseries Mathematics Subject Classification
(2000)}\enspace}
\def\subclass#1{\par\addvspace\medskipamount{\rightskip=0pt plus1cm
\def\and{\ifhmode\unskip\nobreak\fi\ $\cdot$
}\noindent\subclassname\ignorespaces#1\par}}

\begin{document}
\title{Curve configurations in the projective plane and their characteristic numbers}
\author{Adam Czapli\'nski, Piotr Pokora}

\date{\today}
\maketitle
\thispagestyle{empty}
\begin{abstract}
In this paper we study the concept of characteristic numbers and Chern slopes in the context of curve configurations in the real and complex projective plane. We show that some extremal line configurations inherit the same asymptotic invariants, namely asymptotic Chern slopes and asymptotic Harbourne constants which sheds some light on relations between the bounded negativity conjecture and the geography problem for surfaces of general type. We discuss some properties of Kummer extensions, especially in the context of ball-quotients. Moreover, we prove that for a certain class of smooth curve configurations in the projective plane their characterstic numbers are bounded by $8/3$.
\subclass{14C20, 52C35}
\end{abstract}

\section{Introduction}

In the first part of this paper we investigate somewhat extremal properties of line configurations in the real and complex projective plane. Our motivation to study these line configurations is twofold, namely line configurations appear naturally in combinatorics (for instance in the so-called orchard problem \cite{GT}) and it turned out that there is a surprising application of line configurations in the context of the so-called bounded negativity conjecture and the local negativity of surfaces. Let us present a short outline on the mentioned problems.
\begin{definition}
Let $X$ be a smooth projective surface. We say that $X$ has the \emph{bounded negativity property} if there exists an integer $b(X)$ such that for \emph{every reduced curve} $C \subset X$ one has the bound
$$C^{2} \geq -b(X).$$
\end{definition}
The above definition leads to the bounded negativity conjecture \cite{Duke, Harbourne1}, which is one of the most natural and intriguing problems in the theory of algebraic surfaces.
\begin{conjecture}[BNC]
An arbitrary smooth complex projective surface has bounded negativity.
\end{conjecture}
It turns out that the BNC is not true over fields of positive characteristics due to the existence of the Frobenius morphism, which allows to construct easily a counterexample. 

On the other hand, it is easy to construct surfaces with curves having arbitrary low self-intersection numbers. The easiest one is the blow up of the real or complex projective plane along a set of $s >> 0$ collinear points $P_{1}, ..., P_{s}$. Then one has $L_{P_{1}, ..., P_{s}} = H-E_{1} - ... -E_{s}$, where $H$ is the pull-back of $\mathcal{O}_{\mathbb{P}^{2}}(1)$, $E_{1}, ..., E_{s}$ are the exceptional divisors, and $(L_{P_{1}, ..., P_{s}})^{2} = -s+1$. In order to avoid such trivialities one can consider the asymptotic version of the self-intersection numbers, namely in the case of blow ups $f: X_{s} \rightarrow \mathbb{P}^{2}$ of the projective plane along sets $\mathcal{P} = \{P_{1}, ..., P_{s}\}$ of mutually distinct  $s \geq 1$  points we define the following constant
\begin{equation}
\label{eq:Harborne}
H(\mathbb{P}^{2};\mathcal{P}):= \inf \frac{ \left(f^{*}C-\sum_{i=1}^s {\rm mult}_{P_i}C \cdot E_i \right)^2}{s},
\end{equation}
where the infimum is taken over all \emph{reduced} curves $C\subset \mathbb{P}^{2}$.
This constant is called \emph{the local Harbourne constant at $\mathcal{P}$}, and this notion was introduced for the first time in \cite{BdRHHLPSz}.
It turns out that there exists a uniform bound for the so-called \emph{linear Harbourne constants}, which will be denoted in the sequel by $H_{L}$. Instead of considering all possible curves, we restrict the setting of the local Harbourne constants to the strict transforms of \emph{reduced line configurations}.
\begin{theorem}{\cite[Theorem~3.3]{BdRHHLPSz}}
Let $\mathcal{L}$ be a line configuration in $\mathbb{P}^{2}_{\mathbb{C}}$, then one has $$H_{L}(\mathbb{P}_{\mathbb{C}}^{2}; {\rm Sing}(\mathcal{L})) \geq -4,$$
where ${\rm Sing}(\mathcal{L})$ denotes the singular locus of a line configuration $\mathcal{L}$.
\end{theorem}
In fact, the authors have shown a little bit stronger result, namely one can replace sets ${\rm Sing}(\mathcal{L})$ by arbitrary sets of mutually distinct points $\mathcal{P}$. The proof of this result is based on Hirzebruch's inequality. Given a configuration of lines $\mathcal{L}$ then by $t_{k}$ we denote the number of $k$-fold points (points at which exactly $k$ lines from $\mathcal{L}$ intersect).
\begin{theorem}[Hirzebruch's inequality]
Let $\mathcal{L}$ be a configuration of $d \geq 4$ lines in $\mathbb{P}_{\mathbb{C}}^{2}$ such that $t_{d} = t_{d-1} = 0$. Then one has
$$t_{2} + \frac{3}{4}t_{3} \geq d + \sum_{r \geq 5}(r-4)t_{r}.$$
\end{theorem}

Let us now restrict out attention to the case of the real numbers. In 1941 E. Melchior \cite{M41} has shown that for configurations of $k \geq 3$ lines with $t_{k}=0$ in the real projective plane one has the following inequality
$$t_{2} \geq 3 + \sum_{k \geq 4} (k-3)t_{k}.$$
Using the above inequality one can show the following result.
\begin{theorem}{\cite[Theorem~3.15]{BdRHHLPSz}}
\label{th:real}
Let $\mathcal{L}$ be a line configuration in $\mathbb{P}^{2}_{\mathbb{R}}$. Then one has
$$H_{L}(\mathbb{P}^{2}_{\mathbb{R}}; {\rm Sing}(\mathcal{L})) \geq -3.$$
\end{theorem}
It is worth pointing out that Melchior's inequality is related to the famous orchard problem, i.e., using this inequality one sees that the number of \emph{ordinary lines} in the real projective plane is at least $3$ provided that there are at least $3$ non-collinear points.

In general, our expectation is that the local Harbourne constants are bounded from below uniformly.

\begin{conjecture}(Weak Local Negativity)
Let $C \subset \mathbb{P}^{2}_{\mathbb{C}}$ be a reduced curve. Then
$$H(\mathbb{P}^{2}_{\mathbb{C}}; {\rm Sing}(C)) \geq -4.$$
\end{conjecture}

Now let us come back to a short discussion on linear Harbourne constants. Over the complex numbers we do not know whether there exists a line configuration such that $H_{L} \approx -4$. The most negative example is delivered by Wiman's configuration of lines $W_{45}$ \cite{Wiman}. This configuration consists of $45$ lines and has exactly $120$ triple points, $45$ quadruple points, and $36$ quintuple points. Of course $W_{45}$ cannot be realized over the real numbers due to the (dual) Sylvester-Gallai theorem which tells us that every real essential configuration of lines contains at least one double point. Easy computations show that
$$H_{L}(\mathbb{P}^{2}; {\rm Sing}(W_{45})) \approx -3.36.$$
The real case seems to be a little bit more interesting since the bound from Theorem \ref{th:real} is asymptotically achievable (i.e., when the number of lines in configurations $k$ tends to infinity) and, which is also very remarkable, these configurations are strictly related to the orchard problem. 

In this paper we will consider three types of real configurations of lines, namely B\"or\"oczky line configurations, polyhedral (simplicial) line configurations, $s$-elliptic configurations, which have asymptotically the same value of $H_{L}$ and they have the same asymptotical Chern slopes (defined in Section 3), which also sheds some light on new relations between combinatorics of extremal line configurations, the bounded negativity problem and the geography problem of surfaces of general type. These numerical relations are formulated in \emph{Theorem 2.12}. In addition, we show that these three real line configurations share the same asymptotic Chern slopes (and the same value of $H_{L}$) as the so-called Fermat line configurations, which do not have double points and thus they cannot be realized over the real numbers. Also we discuss some properties of the so-called Kummer covers, especially in the context of ball-quotients. In particular, we show explicitly that almost all simplicial and B\"or\"oczky line arrangements do not provide ball-quotients via Hirzebruch's construction -- see \emph{Theorem 2.17}.

In the second part of the paper, we show, in the spirit of Hirzebruch's paper \cite{Hirzebruch}, that for transveral configrations of smooth curves in the complex projective plane their characteristic numbers (one of our Chern slopes defined in the note) are strictly less than $8/3$ -- see \emph{Theorem 3.6}. It is worth pointing out here that in the case of line configurations we have a sharp bound given by $8/3$ due to a result of Sommese \cite{AJS}.
\section{Line configurations and Chern slopes} 
\subsection{Extremal line configurations}
In this section we briefly recall previously mentioned line configurations. The first three configurations are defined over the real numbers. Moreover, B\"or\"oczky and $s$-elliptic line configurations deliver the maximal possible number of triple points for arrangements defined over the real numbers \cite[Theorem 1.3]{GT}
\begin{example}[B\"or\"oczky line configurations]
This construction can be found in \cite{FP84}. We start with a regular $k$-gon inscribed in a circle $O$. We denote vertices of the this $k$-gon by $P_{0}, ..., P_{k-1}$. For simplicity we assume in this note that $k$ is even (of course one can handle with $k$ odd, but it is a little bit more involving). Then we construct the first line by joining $P_{0}$ with $P_{k/2}$. In the next step we join $P_{k/2 - 2}$ with $P_{1}$ and so forth. Of course, it may happen that $P_{k/2 - 2i}$ and $P_{i}$ coincide -- then we draw the tangent line at $P_{i}$ to $O$. We obtain the configuration $\mathcal{B}_{k}$ which consists of $k$ lines, $k-3 +\varepsilon$ double points and $1+ \lfloor \frac{k(k-3)}{6} \rfloor$ triple points, where
$\varepsilon$ is equal to $0$ if $k \equiv 0 \, {\rm mod}(3)$, or $2$ if $k \equiv 1 \, {\rm mod}(3)$ or $k \equiv 2 \, {\rm mod}(3)$
\end{example}
\begin{example}[$s$-elliptic configurations]
\label{s-elliptic}
This example is taken from \cite[p.~120]{Hirzebruch}.
Let $D$ be a smooth plane cubic curve. Pick a flex to define the group law on the points of $D$.
Let $\mathcal{L}$ be the lines dual to the points of a finite group $U$ of order $k$ and let $w$ be the number of
lines in $\mathcal{L}$ whose duals are flex points. Let $\mathcal{E}_{k}$ be the union of the $k$ lines.
Then $t_2=k-w$ and $t_3=\frac{k(k-3)}{6}+\frac{w}{3}$.
\end{example}
\begin{example}[Polyhedral line configurations]
Let $\mathcal{P}_{k}$ be a configuration of $2k$ lines where $k$ of the lines are the sides of a regular $k$-gon
and the other $k$ lines are the lines of bilateral symmetry of the $k$-gon (i.e., angle bisectors and
perpendicular bisectors of the sides). Then it is not hard to check that
$t_{k}=1$ (this is the center of the $k$-gon), $t_{2}=k$ (these are the midpoints of the sides)
$t_{3} = \binom{k}{2}$ (these are the intersections of pairs of sides with the line of symmetry between the sides of the pair)
and $t_{r}=0$ for $r>3$.
\end{example}

\begin{example}[Fermat line configurations]
This example is taken from \cite[Example~II.6.]{GU}
Consider three non-collinear points of the projective space. Then we can construct a configuration $\mathcal{F}_{k}$ of $3k$ lines with $k\geq 3$ (each $k$ lines pass through each point) in such a way that these $3k$ lines intersect exactly in $k^{2}$ triple points. For $k=3$ this is exactly the well-known dual Hesse configuration consisting of $9$ lines and $12$ triple points. In general, we have $t_{k} = 3$, $t_{3}=k^{2}$. Since every $\mathcal{F}_{k}$ does not contain any double point thus it cannot be realized over the real numbers.
\end{example}
\subsection{Asymptotic Chern slopes}
In the celebrated paper \cite{Hirzebruch} Hirzebruch studied the geography problem of surfaces of general type in the context of the Bogomolov-Miyaoka-Yau inequality \cite{M84} for the so-called Kummer extensions.
\begin{theorem}
Let $X$ be a smooth projective surface of general type (i.e. canonical divisor $K_{X}$ is ample) defined over complex numbers. Denote by $c_{1}(X)$ and $c_{2}(X)$ the first and the second Chern class of the complex tangent bundle. Then one has
$$c_{1}^{2}(X) \leq 3c_{2}(X).$$
Moreover, if $c_{1}^{2}(X) = 3c_{2}(X)$, then $X$ is a ball quotient, i.e., the universal cover is the unit complex ball.
\end{theorem}
\begin{remark}
The Bogomolov-Miyaoka-Yau inequality is not true in general for fields of positive characteristics. Counterexamples can be found, for instance, in \cite[Page~142]{BHH87},\cite{Easton2008}, \cite{Szpiro}. Recently Langer has shown that the Bogomolov-Miyaoka-Yau inequality holds for smooth projective surfaces with non-negative Kodaira dimension defined over fields $k$ with the additional assumption $X$ can be lifted to the second Witt vectors $W_{2}(k)$ - we refer to his beautiful paper \cite{Langer} for further details.
\end{remark}
The idea of Hirzebruch was to study surfaces of general type which can be constructed using the Kummer extension. We briefly recall his construction.

Let $\mathcal{L} = \{l_{1}, ..., l_{k}\} \subset \mathbb{P}_{\mathbb{C}}^{2}$ be a configuration of $k \geq 4$ lines such that $t_{k}=0$ and let $n\in\mathbb{Z}_{\geq 2}$.
Now we can consider the Kummer extension of degree $n^{k-1}$ and Galois group $(\mathbb{Z}/n\mathbb{Z})^{k-1}$ defined as the function field
$$\mathbb{C}(z_{1}/z_{0}, z_{2}/z_{0})((l_{2}/l_{1})^{1/n}, ...,(l_{k}/l_{1})^{1/n})$$
which is an abelian extension of the function field of $\mathbb{P}^{2}_{\mathbb{C}}$.
This function field determines an algebraic surface $X_{n}$ with normal singularities which ramifies over the plane with the configuration as the locus of the ramification. Hirzebruch showed that $X$ is singular exactly over a point $p$ iff $p$ is a point of multiplicity $m(p) \geq 3$ in $\mathcal{L}$. After blowing up these singular points we obtain a smooth surface $Y_{n}$. It turns out that $Y_{n}$ is of general type if $t_{k}=t_{k-1} = t_{k-2} = 0$ and $n \geq 2$ or $t_{k}=t_{k-1}=0$ and $n \geq 3$. In the sequel we will called $Y_{n} := Y_{n}^{\mathcal{L}}$ as \emph{the Kummer cover} of order $n^{k-1}$.

The natural arising question is whether it is possible to construct a ball quotient using Kummer extensions.
\begin{example}
\label{ex:ballquotient}
Let $n=5$ and consider the following real line configuration:
$$k = 6, t_{2} = 3, t_{3} = 4.$$
Then $Y_{5}$ is a ball quotient.
\end{example}
Recall that for a line configuration $\mathcal{L}$ and $i \in \{0,1,2\}$ one defines
$$f_{i} = \sum_{r\geq 2}r^{i}t_{r}.$$
Summing up the most important observations from Hirzebruch's paper \cite{Hirzebruch}, we get the following theorem. 
\begin{theorem}
Let $Y_{n}^{\mathcal{L}}$ be the Kummer cover of order $n^{k-1}$, then one has
$$\frac{c_{2}(Y_{n}^{\mathcal{L}})}{n^{k-3}} = n^{2}(3-2k+f_{1}-f_{0}) + 2n(k-f_{1}+f_{0}) + f_{1}-t_{2},$$
$$\frac{c_{1}^{2}(Y_{n}^{\mathcal{L}})}{n^{k-3}} = n^{2}(-5k+9+3f_{1}-4f_{0}) + 4n(k-f_{1}+f_{0}) +f_{1}-f_{0}+k+t_{2}.$$
\end{theorem}
\begin{definition}
Let $X$ be a smooth minimal complex projective surfaces of general type. The number
$$\frac{c_{1}^{2}(X)}{c_{2}(X)} \in [1/5, 3]$$
is called \emph{the Chern slope of $X$}.
\end{definition}
We refer to \cite{RU15} for very recent results devoted to Chern slopes.

If $\mathcal{L}_{k}$ is a sequence of line configurations (like in the case of B\"or\"oczky line configuration etc.) it is natural to ask how the Chern slope for the Kummer cover $Y_{n}^{\mathcal{L}_{k}}$ behaves asymptotically with $n; k \rightarrow \infty$.
\begin{definition}
Let $Y_{n}^{\mathcal{L}_{k}}$ be the Kummer cover of order $n^{k-1}$ associated to a sequence of line configurations $\mathcal{L}_{k}$.
Then
\begin{itemize}
\item the $k$-Chern slope is defined as
$${\rm Ch}_{k}(Y_{n}^{\mathcal{L}_{k}}) = \lim_{k\rightarrow \infty} \frac{ c_{1}^{2}(Y_{n}^{\mathcal{L}_{k}})}{c_{2}(Y_{n}^{\mathcal{L}_{k}})} \in \mathbb{C}[n].$$
\item the $n$-Chern slope (or the characteristic polynomial of $\mathcal{L}_{k}$) is defined as
$${\rm Ch}_{n}(Y_{n}^{\mathcal{L}_{k}}) = \lim_{n\rightarrow \infty} \frac{ c_{1}^{2}(Y_{n}^{\mathcal{L}_{k}})}{c_{2}(Y_{n}^{\mathcal{L}_{k}})} \in \mathbb{C}[k].$$
\item the $(k,n)$-Chern slope is defined as
$${\rm Ch}_{(k,n)}(Y_{n}^{\mathcal{L}_{k}}) =\lim_{n, k\rightarrow \infty} \frac{ c_{1}^{2}(Y_{n}^{\mathcal{L}_{k}})}{c_{2}(Y_{n}^{\mathcal{L}_{k}})} \in \mathbb{R}.$$
\end{itemize}
\end{definition}
\begin{definition}
Let $\mathcal{L}_{k} \subset \mathbb{P}^{2}$ be a sequence of line configurations of $k$ lines. Then the asymptotic linear Harbourne constant at ${\rm Sing}(\mathcal{L}_{k})$ is defined as
$$\widehat{H_{L}(\mathcal{L}_{k})} = \lim_{k \rightarrow \infty} H_{L}(\mathbb{P}^{2} ; {\rm Sing}(\mathcal{L}_{k})).$$
\end{definition}
Now we are ready to formulate our first result.
\begin{theorem}
\label{thm:geometric}
Under the notions as above, we have the following equalities:
\begin{enumerate}
\item $\widehat{H_{L}(\mathcal{B}_{k})} = \widehat{H_{L}(\mathcal{E}_{k})} = \widehat{H_{L}(\mathcal{P}_{k})} = \widehat{H_{L}(\mathcal{F}_{k})} = -3$,

\item ${\rm Ch}_{k}(Y_{n}^{\mathcal{B}_{k}}) = {\rm Ch}_{k}(Y_{n}^{\mathcal{E}_{k}}) = {\rm Ch}_{k}(Y_{n}^{\mathcal{P}_{k}}) = {\rm Ch}_{k}(Y_{n}^{\mathcal{F}_{k}}) = \frac{\frac{5}{2}n^{2} - 4n + 1}{n^{2}-2n+\frac{3}{2}},$
\item ${\rm Ch}_{(k,n)}(Y_{n}^{\mathcal{B}_{k}}) = {\rm Ch}_{(k,n)}(Y_{n}^{\mathcal{E}_{k}}) = {\rm Ch}_{(k,n)}(Y_{n}^{\mathcal{P}_{k}}) = {\rm Ch}_{(k,n)}(Y_{n}^{\mathcal{F}_{k}}) = \frac{5}{2}.$
\end{enumerate}
\end{theorem}
\begin{proof}
$\text{Ad } 1.)$ Observe that the local linear Harbourne constant of $\mathcal{L}_{k}$ at ${\rm Sing}(\mathcal{L}_{k})$ has the form
$$H_{L}(\mathbb{P}^{2}; {\rm Sing}(\mathcal{L}_{k})) = {\rm inf}_{ \mathcal{L}_{k}} \frac{ k^{2}  - f_{2}}{f_{0}} = {\rm inf}_{ \mathcal{L}_{k}} \frac{k-f_{1}}{f_{0}},$$
where the last equality is the consequence of the well-know combinatorial equality $k^{2} - k = f_{2} - f_{1}$.
Then for $\mathcal{F}_{k}$ one has
$$\widehat{H_{L}(\mathcal{F}_{k})} = \lim_{k \rightarrow \infty} \frac{-3k^{2}}{k^{2}+3} = -3.$$
For other configurations computations are similar. \\\\
$\text{Ad } 2.)$
Quite tedious computations lead to the following equalities:

$$c_{2}(Y_{n}^{\mathcal{B}_{k}})/n^{k-3}= n^{2}\bigg(2+\varepsilon -k + 2 \left\lfloor\frac{k(k-3)}{6} \right\rfloor\bigg) + 2n\bigg(1-\varepsilon-2 \left\lfloor\frac{k(k-3)}{6}\right\rfloor\bigg)+k+ \varepsilon + 3 \left \lfloor\frac{k(k-3)}{6}\right\rfloor,$$

\begin{multline*}
c_{1}^{2}(Y_{n}^{\mathcal{B}_{k}})/n^{k-3} = n^{2}\bigg(-3k+8 + 2\varepsilon +5 \left\lfloor\frac{k(k-3)}{6}\right\rfloor\bigg)+4n\bigg(1-\varepsilon-2\left\lfloor\frac{k(k-3)}{6}\right\rfloor\bigg) + \\ + 3k-4+2\varepsilon + 2\left\lfloor\frac{k(k-3)}{6}\right\rfloor,
\end{multline*}
$$c_{2}(Y_{n}^{\mathcal{E}_{k}})/n^{k-3} = n^{2}\bigg(3-k-\frac{w}{3} + \frac{2k(k-3)}{6} \bigg) + 2n\bigg(\frac{w}{3}-\frac{2k(k-3)}{6}\bigg) + k+\frac{k(k-3)}{2},$$
$$c_{1}^{2}(Y_{n}^{\mathcal{E}_{k}})/n^{k-3} = n^{2} \bigg( -3k-\frac{w}{3}+9 + \frac{5k(k-3)}{6} \bigg) + 4n\bigg(\frac{w}{3}-\frac{2k(k-3)}{6} \bigg) + 3k - \frac{4w}{3} + \frac{k(k-3)}{3}, $$
$$c_{2}(Y_{n}^{\mathcal{P}_{k}})/n^{k-3} = n^{2} (2-3k+k^{2}) + 2n(1-k^{2}) + \frac{3}{2}k^{2} + \frac{1}{2}k,$$
$$c_{1}^{2}(Y_{n}^{\mathcal{P}_{k}})/n^{k-3} = n^{2} \bigg(5+\frac{5}{2}(k^{2}-3k)\bigg) + 4n(1-k^{2}) + k^{2} +3k -1,$$
$$c_{2}(Y_{n}^{\mathcal{F}_{k}})/n^{k-3} = n^{2} (-3k+2k^{2} )+2n(3-2k^{2}) + 3k^{2}+3k,$$
$$c_{1}^{2}(Y_{n}^{\mathcal{F}_{k}})/n^{k-3} = n^{2} (5k^{2}-6k-3) + 4n(3-2k^{2}) +2k^{2}+4k-3.$$
Then taking limits with $k\rightarrow \infty$ one gets the common polynomial $$\frac{\frac{5}{2}n^{2}-4n+1}{n^{2}-2n+\frac{3}{2}}.$$\\
$\text{Ad } 3.)$ This is a consequence of $\text{Ad } 2).$ Observe that for our extremal line configurations we have
$$\lim_{k \rightarrow \infty} \lim_{n \rightarrow \infty}  {\rm Ch}(Y_{n}^{\mathcal{L}_{k}}) = \lim_{n \rightarrow \infty} \lim_{k \rightarrow \infty} {\rm Ch}(Y_{n}^{\mathcal{L}_{k}}).$$
\end{proof}
\begin{remark}
In \cite{Hirzebruch} Hirzebruch introduced for a line configuration $\mathcal{L}$ its characteristic number $\gamma$, which is defined as
$$\gamma = \lim_{n \rightarrow \infty} \frac{c_{1}^{2}(Y_{n}^{\mathcal{L}})}{c_{2}(Y_{n}^{\mathcal{L}})},$$
and he observed that
$$\gamma = \frac{5}{2} - \frac{3f_{0}-f_{1}-3}{2(3-2k+f_{1}-f_{0})}$$
if only $3-2k+f_{1}-f_{0}>0$.

In \cite{AJS} Sommese has shown that for a line configuration $\mathcal{L}$ of $k \geq 6$ lines one has the following inequalities
$$ 2 \bigg( \frac{k-3}{k-2} \bigg) \leq \gamma \leq \frac{8}{3},$$
where the left-hand side inequality becomes equality for general line configurations (possessing only double points as singularities) and the right-hand side inequality becomes  equality for the dual Hesse-configuration.
\end{remark}
Theorem \ref{thm:geometric} presents a certain phenomenon, which can be further generalized.
\begin{corollary}
Suppose that  $\mathcal{L}_{k} \subset \mathbb{P}^{2}$ is a sequence of line configurations such that
$$t_{3} = ck^{2} + \mathcal{O}(k) \text{ with } c>0, $$
$$t_{2}, t_{4}, ..., t_{k} \text{ have at most linear growth}.$$
Then $\widehat{H_{L}(\mathcal{L}_{k})} = -3$, ${\rm Ch}_{k}(Y_{n}^{\mathcal{L}_{k}}) = \frac{\frac{5}{2}n^{2}-4n+1}{n^{2}-2n+\frac{3}{2}}$ and ${\rm Ch}_{(n,k)}(Y_{n}^{\mathcal{L}_{k}}) = \frac{5}{2}.$
\end{corollary}
\begin{proof}
We have
$$\widehat{H_{L}(\mathcal{L}_{k})} = \lim_{k \rightarrow \infty}\frac{k - 3ck^{2} - \mathcal{O}(k)}{ck^2 + \mathcal{O}(k)} = -3.$$
In a similar way one sees that for $k >> 0$
$$c_{2}(Y_{n}^{\mathcal{L}_{k}})/n^{k-3} \approx 2ck^{2}n^{2} - 4cnk^{2} + 3ck^{2},$$
$$c_{1}^{2}(Y_{n}^{\mathcal{L}_{k}})/n^{k-3} \approx 5ck^{2}n^{2} - 8cnk^{2} + 2ck^{2}$$
which completes the proof.
\end{proof}
\begin{remark}
Erd\"os and Purdy have shown that if $\mathcal{L} \subset \mathbb{P}^{2}_{\mathbb{R}}$ is a configuration of at least $k \geq 3$ lines with $t_{k}=0$ and $t_{2} < k-1$, then there exists a positive constant $c$ such that $t_{3} \geq ck^{2}$.
\end{remark}
\begin{corollary}
Suppose that  $\mathcal{L}_{k} \subset \mathbb{P}^{2}$ is a sequence of line configurations such that
$$t_{4} = ck^{2} +\mathcal{O}(k) \text{ with } c>0$$
$$t_{2}, t_{3}, t_{5}, ..., t_{k} \text{ have at most linear growth}.$$
Then $\widehat{H_{L}(\mathcal{L}_{k})} = -4$, ${\rm Ch}_{k}(Y_{n}^{\mathcal{L}_{k}}) = \frac{\frac{8}{3}n^{2}-4n+1}{n^{2}-2n+\frac{4}{3}}$ and ${\rm Ch}_{(n,k)}(Y_{n}^{\mathcal{L}_{k}}) = \frac{8}{3}.$
\end{corollary}
At this moment we do not know whether there exists a sequence of line configurations in $\mathbb{P}^{2}_{\mathbb{C}}$ such that the number of quadruple points has growth $ck^{2} + \mathcal{O}(k)$ with $c>0$, but the above corollary shows that if such configuration exists, then ${\rm Ch}_{k}(Y_{3}^{\mathcal{L}_{k}}) = 3$, which suggests a very rigid structure of $Y_{3}^{\mathcal{L}_{k}}$. Observe that such a sequence of line configurations does not exist over the real numbers, since if we could construct such a sequence $\mathcal{L}_{k}'$, then for instance  $\widehat{H_{L}(\mathcal{L}_{k}')} = -4$, a contradiction with Theorem \ref{th:real}.
\subsection{Ball quotients and Kummer covers}
Let us come back to the configuration from Example \ref{ex:ballquotient}. It is worth pointing out that this configuration can be realized both as $\mathcal{P}_{3}$ and as $\mathcal{B}_{6}$, for simplicity let us denote it by $\mathcal{T}_{6}$. Now we calculate $c_{1}^{2}(Y_{n}^{\mathcal{T}_{6}})$ and $c_{2}(Y_{n}^{\mathcal{T}_{6}})$, namely
$$c_{2}(Y_{n}^{\mathcal{T}_{6}})/n^{3} = 2n^{2} -10n + 15,$$
$$c_{1}^{2}(Y_{n}^{\mathcal{T}_{6}})/n^{3} = 5n^{2} - 20n + 20.$$
Observe that for $n=5$ we obtain
$$\frac{c_{1}^{2}(Y_{n}^{\mathcal{T}_{6}})}{c_{2}(Y_{n}^{\mathcal{T}_{6}})} = 3,$$
for $n \geq 5$ the sequence of slopes is decreasing with
$$\lim_{n \rightarrow  \infty} \frac{c_{1}^{2}(Y_{n}^{\mathcal{T}_{6}})}{c_{2}(Y_{n}^{\mathcal{T}_{6}})} = \frac{5}{2}.$$
On the other hand, plugging data into Hirzebruch's inequality, we get
$$t_{2} + \frac{3}{4}t_{3} = 3 + \frac{3}{4}\cdot 4 \geq k + \sum_{r\geq 5}(r-4)t_{r} = 6.$$
Moreover, if we use the same data to the Melchior's inequality, then we get
$$t_{2} = 3 \geq 3 + \sum_{r\geq 4} (r-3)t_{r} = 3.$$

\begin{theorem}
For $k > 6$, $\ell \geq 4$ and an arbitrary $n \in \mathbb{Z}_{\geq 2}$ Kummer covers $Y_{n}^{\mathcal{B}_{k}}$ and $Y_{n}^{\mathcal{P}_{\ell}}$ are never ball-quotients.
\end{theorem}
\begin{proof}
Let us define (via \cite[Section 3.1]{Hirzebruch}) the quadratic polynomial of a configuration $\mathcal{L}_{k}$
$$P_{\mathcal{L}_{k}}(n) = \frac{3c_{2}(Y_{n}^{\mathcal{L}_{k}}) - c_{1}^{2}(Y_{n}^{\mathcal{L}_{k}})}{n^{k-3}} = n^{2}(f_{0}-k) + 2n(k-f_{1}+f_{0})+2f_{1}+f_{0}-k-4t_{2}.$$
Hirzebruch showed that for a line configuration with $t_{k}=t_{k-1}=t_{k-2} = 0$ the quadratic polynomial $P_{\mathcal{L}_{k}}(n)$ is non-negative for every integer $n \in \mathbb{Z}$. We need to show that $P_{\mathcal{B}_{k}}(n) > 0$ for every $n \geq 2$, $k \geq 8$ and $P_{\mathcal{P}_{k}}(n) >0$ for $n \geq 2$, $k \geq 4$. After a simple change of coordinates $x \rightarrow n+1$ we can express $P_{\mathcal{L}_{k}}(n)$ as
$$P_{\mathcal{L}_{k}}(x) = x^{2}(f_{0}-k)-2x(f_{1}-2f_{0})+4(f_{0}-t_{2}).$$
For $\mathcal{P}_{k}$ we have
$$P_{\mathcal{P}_{k}}(x) =x^{2}\frac{(k-1)(k-2)}{2} - x(k^{2}+k-4) + 2k^{2}-2k+4.$$
Now it is easy to observe that $P_{\mathcal{P}_{k}}(x)$ has a minimum at $k = 4, x=3$ with $P_{\mathcal{P}_{4}}(3) = 7$, which completes the proof for $\mathcal{P}_{k}$.

The quadratic polynomial for $\mathcal{B}_{k}$ has the form
$$P_{\mathcal{B}_{k}}(x) =x^{2} \bigg( \varepsilon - 2 + \left\lfloor \frac{k(k-3)}{6} \right\rfloor \bigg) - 2x \bigg(1 + \left\lfloor \frac{k(k-3)}{6} \right\rfloor \bigg) +4 \bigg( 1+ \left\lfloor \frac{k(k-3)}{6}\right\rfloor\bigg).$$
Again, it is easy to observe that $P_{\mathcal{B}_{k}}(x)$ has a minimum at $k=8$, $x=1$ with $P_{\mathcal{B}_{8}}(1) = 20$.

\end{proof}
Let us now consider Klein's configuration of lines \cite{Kle79}, which will be denoted in the sequel by $\mathcal{K}_{21}$. Firstly, $\mathcal{K}_{21}$ consists of $21$ lines and has exactly $28$ triple points and $21$ quadruple points. Moreover, this configuration cannot be realized over the real numbers. Observe that plugging this data into Hirzebruch's inequality one has
$$t_{2} + \frac{3}{4}t_{3} = 28\cdot \frac{3}{4} \geq k + \sum_{r\geq 5} (r-4)t_{r} = 21 + 0 \cdot 21.$$
Let us now compute the Chern numbers for $Y_{n}^{\mathcal{K}_{21}}$. Again, some computations provide
$$c_{1}^{2}(Y_{n}^{\mathcal{K}_{21}})/n^{18} = 212n^{2}-392n+140,$$
$$c_{2}(Y_{n}^{\mathcal{K}_{21}})/n^{18} = 80n^{2}-196n+168.$$
Firstly,
$$ {\rm Ch}_{n}(Y_{n}^{\mathcal{K}_{21}}) = 2.65,$$
and for $n=4$ we obtain the maximal value of the Chern slope, which is equal to
$$\frac{c_{1}^{2}(Y_{4}^{\mathcal{K}_{21}})}{c_{2}(Y_{4}^{\mathcal{K}_{21}})} = 2.95783.$$
This example shows that being a ball quotients is a very delicate property. It would be quite interesting to detect some other invariants which allow to decide whether a certain configuration of lines delivers a ball quotient as the Kummer extension.

\section{$d$-configurations and their characteristic numbers}
We start with the following definition.
\begin{definition}
Let $\mathcal{C} = \{C_{1}, ...,C_{k}\} \subset \mathbb{P}^{2}_{\mathbb{C}}$ be a configuration of $k\geq 4$ curves. We say that $\mathcal{C}$ is a $d$-configuration if
\begin{itemize}
\item all curves are smooth of degree $d\geq 2$,
\item all intersection points are transveral,
\item $t_{k}=0$.
\end{itemize}
\end{definition}

For such configurations of curves we have the following combinatorial equality 
\begin{equation}
\label{eq:comb}
d^{2} (k^{2}-k) = \sum_{r \geq 2} (r^{2} - r)t_{r}.
\end{equation}
We say that a $d$-configuration is \emph{general} if $t_{r} = 0$ for $r \geq 3$.

The key observation is that one can mimic Hirzebruch's argumentation in our setting of $d$-configurations \cite{PRSz}.
The first step is to construct an abelian cover $X_{n}$ of order $n^{k-1}$ branched along a given $d$-configuration and after the minimal desingularization we obtain the surface $Y_{n}^{\mathcal{C}}$ having the following characteristic numbers

$$\frac{c_{2}(Y_{n}^{\mathcal{C}})}{n^{k-3}} = n^{2}(3 +(d^{2}-3d)k +f_{1}-f_{0})+n(-(d^{2}-3d)k-2f_{1}+2f_{0})+f_{1}-t_{2},$$
\begin{multline*}
\frac{c_{1}^{2}(Y_{n}^{\mathcal{C}})}{n^{k-3}} = n^{2}( 9+d^{2}k-6dk + 3f_{1} - 4f_{0}) +2n(-(d^{2}-3d)k-2f_{1}+2f_{0}) + \\ + 3dk + (d^{2}-3d)k+f_{1}-f_{0}+t_{2}.
\end{multline*}

Under the assumption that $t_{k}=0$ and $n\geq 2$ one can show that $Y_{n}^{\mathcal{C}}$ has non-negative Kodaira dimension (in fact the canonical divisor is big and nef) and we can use the Bogomolov-Miyaoka-Yau inequality \cite{M77}, namely
$$c_{1}^{2}(Y_{n}^{\mathcal{C}}) \leq 3c_{2}(Y_{n}^{\mathcal{C}}).$$
This allows us to deduce the following Hirzebruch-type inequality for $d$-configurations with $d\geq 3$.

\begin{theorem}(\cite{PRSz})
\label{thm:hirz} Let $\mathcal{C} \subset \mathbb{P}^{2}_{\mathbb{C}}$ be a $d$-configuration, then
$$\bigg(\frac{7}{2}d^{2}-\frac{9}{2}d\bigg)k + t_{2} + t_{3} \geq \sum_{r \geq 4} (r-4)t_{r}.$$
\end{theorem}
For configurations of smooth conics we have the following result due to Tang \cite[Theorem 3.1 with $x=2$]{LT}.
\begin{theorem}
\label{thm:tang}
Let $\mathcal{C}$ be a $2$-configuration, then
$$5k + t_{2} + t_{3} \geq \sum_{r\geq 4}(r-4)t_{r}.$$
\end{theorem}

In the spirit of Hirzebruch's paper we define the following characteristic number of a $d$-configuration $\mathcal{C}$.
\begin{definition}
Let $\mathcal{C}$ be a $d$-configuration of curves. Then
$$\gamma(\mathcal{C}) = \lim_{n \rightarrow \infty} \frac{c_{1}^{2}(Y_{n}^{\mathcal{C}})}{c_{2}(Y_{n}^{\mathcal{C}})} = \frac{ 9+d^{2}k-6dk + 3f_{1} - 4f_{0} }{ 3 +(d^{2}-3d)k +f_{1}-f_{0} }$$ is the characteristic number of $\mathcal{C}$ provided that $c_{2}(Y_{n}^{\mathcal{C}}) > 0$.
\end{definition}
\begin{remark}
In his PhD thesis \cite{GU} G. Urz\'ua considered the so-called logarithmic Chern slopes for logarithmic surfaces. It turns out that these logarithmic Chern slopes coincide in our setting with characteristic numbers of $d$-configurations.
\end{remark}
Now we are ready to present our result of this section.
\begin{theorem}
\label{thm:chern}
Let $\mathcal{C}\subset \mathbb{P}^{2}_{\mathbb{C}}$ be a $d$-configuration. Then
$$\gamma(\mathcal{C}) < \frac{8}{3}.$$
\end{theorem}
\begin{proof}
We show that $3 + d(d-3)k + f_{1} - f_{0} > 0$. This is obviously true for $d\geq 3$ without any argument. Let us now focus on the case $d=2$. Since for $n\geq 2$ surfaces $Y_{n}^{\mathcal{C}}$ are of general type, thus $c_{2}(Y_{n}^{\mathcal{C}}) > 0$ and this implies $f_{1} - f_{0}\geq 2k-3$. Therefore, the linear coefficient of $c_{2}(Y_{n}^{\mathcal{C}})$ is negative and hence $f_{1}-f_{0} > 2k-3$, which completes the first part.

Suppose that $\gamma \geq \frac{8}{3}$. Then we have

$$3d^{2}k-18dk+27+9f_{1}-12f_{0} \geq 24 + 8d^{2}k-24dk+8f_{1}-8f_{0},$$
which leads to 
$$3 + 6dk -5d^{2}k+f_{1}-4f_{0} \geq 0.$$
Now we need to consider two cases, namely
\begin{enumerate}
\item $d \geq 3$. Rewriting the above inequality we have
$$3 + 6dk -5d^{2}k + \sum_{r \geq 4}(r-4)t_{r} \geq 2t_{2} + t_{3} \stackrel{(\star)}{\geq}  t_{2} - \frac{7}{2}d^{2}k + \frac{9}{2}dk + \sum_{r \geq 4}(r-4)t_{r},$$
where the inequality $(\star)$ follows directly from Theorem \ref{thm:hirz}.
We obtain
$$2t_{2} \leq 6 + 3dk(1-d) <0.$$
On the other hand, we know that $t_{2} \geq 0$, a contradiction.
\item $d = 2$. We have
$$3 - 8k + \sum_{r \geq 4}(r-4)t_{r} \geq 2t_{2} + t_{3}.$$
Using Theorem \ref{thm:tang} in $(!)$ we obtain
$$3 - 8k + \sum_{r \geq 4}(r-4)t_{r} \stackrel{(!)}{\geq} t_{2} - 5k + \sum_{r\geq 4}(r-4)t_{r}.$$
This leads to
$$t_{2} \leq 3-3k <0,$$
a contradiction.
\end{enumerate}
\end{proof}
Let us present some examples.
\begin{example}
Let $\mathcal{C}\subset \mathbb{P}^{2}$ be a general $d$-configuration with $d\geq 2$. By the combinatorial equality we have $2t_{2} = d^{2}(k^{2}-k)$ and 
$$\gamma(\mathcal{C}) = 2 \cdot \frac{(dk-3)^{2}}{d^{2}k^{2}+(d^{2}-6d)k+6}.$$
In particular, $\gamma = 2$ provided that $k \rightarrow \infty$, independently of $d$.
\end{example}

\begin{example}
Consider the following conic configuration $\mathcal{AP}$ of $6$ conics and $t_{5}=6$. 
Then
$$\gamma(\mathcal{AP}) = \frac{27}{15}.$$
\end{example}
\begin{example}
In \cite{AD} Dolgachev and Artebani constructed the Hesse configuration of conics $\mathcal{H}$, which consists of $12$ conics with $t_{2} = 12$ and $t_{8} = 9$.

Then $$\gamma(\mathcal{H}) = \frac{13}{6}.$$
\end{example}
It is worth pointing out here that the theory of $d$-configurations is substantially more complicated comparing with the theory of line configurations and we found only several papers strictly devoted to them.

At the end of the paper, we would like to address two quite challenging problems. 
\begin{question}
Does there exist a family $\{\Delta_{n}\}$ of $d$-configurations for fixed $d \geq 2$ such that $\gamma(\Sigma_{n})$ is coverging to $\alpha > 5/2$ provided that $n \rightarrow \infty$?
\end{question}
An analogous question for line configurations is open. It is expected that we have only finitely many line configurations for which the characteristic numbers are contained in the region $(5/2, 8/3]$. This also suggest the following really challenging conjecture.
\begin{conjecture}
For a given $\varepsilon > 0$ and fixed $d \geq 2$ there exist only finitely many $d$-configurations $\mathcal{C}$ such that $\gamma(\mathcal{C}) > 2.5 + \varepsilon$.
\end{conjecture}
\section*{Acknowledgement}
The second author is partially supported by National Science Centre Poland Grant 2014/15/N/ST1/02102 and during the project he was a member of SFB $45$ "Periods, moduli spaces and arithmetic of algebraic varieties". Both authors would like to express their gratitude to Stefan M\"uller-Stach and Duco van Straten for stimulating conversations, and the second author would like to thank Roberto Laface for useful remarks and Giancarlo Urz\'ua for providing his insights about logarithmic Chern slopes. A part of this note was developed during the mini-Semester on Algebraic Geometry in Warsaw supported by the Simons Foundation. At last, we would like to thank the anonymous referee for useful suggestions, especially for those devoted to characteristic numbers.



\bigskip
   Piotr Pokora,
   Instytut Matematyki,
   Pedagogical University of Cracow,
   Podchor\c a\.zych 2,
   PL-30-084 Krak\'ow, Poland.

Current Address:
    Institut f\"ur Mathematik,
    Johannes Gutenberg-Universit\"at Mainz,
    Staudingerweg 9,
    D-55099 Mainz, Germany.
\nopagebreak
   \textit{E-mail address:} \texttt{piotrpkr@gmail.com, pipokora@uni-mainz.de}

\bigskip
    Adam Czapli\'nski,
    Institut f\"ur Mathematik,
    Johannes Gutenberg-Universit\"at Mainz,
    Staudingerweg 9,
    D-55099 Mainz, Germany.
\nopagebreak
   \textit{E-mail address:} \texttt{czaplins@uni-mainz.de}


\end{document}